\documentclass[10pt]{article}

\usepackage{graphicx}
\usepackage[small]{caption}
\usepackage{amsthm}
\usepackage{amsmath}
\usepackage{amsfonts}
\usepackage{amssymb}
\usepackage[usenames]{color}
\usepackage{hyperref}

\usepackage{mathpazo}	

\renewcommand{\baselinestretch}{1.2}

\theoremstyle{plain}
\newtheorem{theorem}{Theorem}[section]
\newtheorem{lemma}[theorem]{Lemma}
\newtheorem{proposition}[theorem]{Proposition}
\newtheorem{corollary}[theorem]{Corollary}

\newcounter{myBackupProp}

\theoremstyle{definition}
\newtheorem{mainthm}[theorem]{Main Theorem}
\newtheorem{definition}[theorem]{Definition}
\newtheorem{example}[theorem]{Example}

\newtheorem{remark}[theorem]{Remark}
\newtheorem{question}[theorem]{Question}
\newtheorem{problem}[theorem]{Problem}
\newtheorem{conjecture}[theorem]{Conjecture}

\def\surjection{\twoheadrightarrow}

\def\Tor{\mathrm{Tor}}		
\def\Aut{\mathrm{Aut}}		

\def\Z{\mathbf{Z}}		
\def\Q{\mathbf{Q}}		
\def\R{\mathbf{R}}		

\def\Res{\mathrm{Res}}		

\def\HH{\mathrm{H}}		
\def\SL{\mathrm{SL}}		
\def\PSL{\mathrm{PSL}}		
\def\St{\mathrm{St}}		
\def\fsys{\mathcal{F}}		
\def\surface{\Sigma}	
\def\vcd{\mathrm{vcd}}	

\def\CC{\mathcal{C}}		
\def\Mod{\mathrm{Mod}}		
\def\Out{\mathrm{Out}}		

\def\Teich{\mathcal{T}}		
\def\AC{\mathcal{A}}		
\def\TB{\Delta}		

\title{Homology of the curve complex and the Steinberg module of the mapping class group}
\author{Nathan Broaddus}
\date{November 3, 2011}

\begin{document}
\maketitle

\begin{abstract}
By the work of Harer, the reduced homology of the complex of curves is a fundamental cohomological object associated to all torsion free finite index subgroups of the mapping class group.  We call this homology group the {\em Steinberg module} of the mapping class group.  It was previously known that the curve complex has the homotopy type of a bouquet of spheres.  Here, we give the first explicit homologically nontrivial sphere in the curve complex and show that under the action of the mapping class group, the orbit of this homology class generates the reduced homology of the curve complex.
\end{abstract}

\tableofcontents



\section{Introduction}
\label{section:introduction}

An {\em orientable Poincar\'{e} duality group} $\Gamma$ of dimension $d$ is a group of type $FP$ whose homology and cohomology are associated via the strong property that
\begin{equation} 
\label{eqn:poincare}
\HH^i(\Gamma; A) \cong \HH_{d - i}(\Gamma; A) \ \ \   \mbox{for any $\Gamma$-module $A$ and $i \in \Z$.}
\end{equation}
Unfortunately, many important groups do not even have finite index subgroups which are orientable Poincar\'{e} duality groups.  For example,  free groups on two or more generators, fundamental groups of nontrivial knot complements, braid groups on three or more strands, $\SL(n,\Z)$ for $n \geq 2$ and mapping class groups of surfaces of genus $g \geq 1$ do not have finite index subgroups which are orientable Poincar\'{e} duality groups.  However, each of these groups does have a finite index subgroup which is a {\em Bieri-Eckmann duality group} \cite{BE1973} (see also \cite[\S 6.1]{Ivanov2002}, \cite[\S VIII.10]{Brown}, \cite[\S 9]{Bieri1976}).

A group $\Gamma$ of type $FP$ is a Bieri-Eckmann duality group, or simply a {\em duality group}, of dimension $d$ if there is a $\Gamma$-module $D$ such that for any $\Gamma$-module $A$ and any $i \in \Z$ we have
\begin{equation}
\label{eqn:bedual}
 \HH^i(\Gamma; A) \cong \HH_{d - i}(\Gamma; D \otimes_{\Z} A).
\end{equation}
Here $\Gamma$ acts on $D \otimes_{\Z} A$ with the diagonal action: $\gamma (x \otimes a) = (\gamma x) \otimes (\gamma a)$ for all $\gamma \in \Gamma$, $x \in D$, and $a \in A$. 
This $\Gamma$-module $D$ is  determined by the duality group $\Gamma$ and is called the {\em dualizing module} of $\Gamma$.  Thus the dualizing module is a fundamental cohomological object associated to any duality group.  It is, in fact, expressible as the top nontrival cohomology group of $\Gamma$ with coefficients in the group ring $\Z \Gamma$.  Hence the dualizing module of the duality group $\Gamma$ with cohomological dimension $d$ is the $\Gamma$-module
\begin{equation*} D \cong \HH^d(\Gamma; \Z \Gamma).
\end{equation*}

 Orientable Poincar\'{e} duality groups may be characterized as duality groups for which the dualizing module is the trivial module $D = \Z$.  In this case Equation~(\ref{eqn:bedual}) simplifies to Equation~(\ref{eqn:poincare}).  Therefore, one may view the dualizing module of $\Gamma$ as a sort of ``error term'' measuring the extent to which a duality group fails to be a Poincar\'{e} duality group.

Let $\Gamma$ be a group with finite index subgroups $\Gamma_1$ and $\Gamma_2$ of finite cohomological dimensions $d_1$ and $d_2$ respectively.  Then $d_1 = d_2$ and $\Gamma$ is said to have {\em virtual cohomological dimension} $\vcd (\Gamma) = d_1$.  

There is a similar invariance of dualizing modules for finite index subgroups.  Let $\Gamma$ be a group with finite index subgroups $\Gamma_1$ and $\Gamma_2$ which are duality groups with dualizing modules  $D_1$ and $D_2$ respectively.  Then $D_1 \cong D_2$ as $(\Gamma_1 \cap \Gamma_2)$-modules.  That is,
\begin{equation*}
\Res^{\Gamma_1}_{\Gamma_1 \cap \Gamma_2}( D_1 )\cong \Res^{\Gamma_2}_{\Gamma_1 \cap \Gamma_2}(D_2).
\end{equation*}

Let $\surface$ denote either the closed genus $g$ surface $\surface_g$ or that surface $\surface_g^1$ with a marked point. Let $\Mod(\surface)$ be the mapping class group of $\surface$. Harvey's \cite[\S 2]{Harvey1981} complex of curves $\CC(\surface)$ (see Definition~\ref{def:cc} below) is a simplicial complex with a natural simplicial action of $\Mod(\surface)$.  Harer has shown that the curve complex has the homotopy type of a wedge sum of spheres of dimension $2g-2$.  The {\em Steinberg module} is the $\Mod(\surface)$-module
\begin{equation*}
\St (\Sigma) = \HH_{2g-2}(\CC(\surface);\Z).
\end{equation*}

Harer \cite[Theorem~4.1]{Harer1986} and Ivanov \cite[Theorem 6.6]{Ivanov1987} in the case of the closed surface have shown that the mapping class group $\Mod(\surface)$ is a virtual duality group, and that the dualizing module for any torsion-free, finite index subgroup of $\Mod(\surface)$ is the Steinberg module $\St(\surface)$.  Thus this module is a fundamental cohomological object associated to the commensurability class of $\Mod(\surface)$.

\subsection{Summary of results}
Again, let $\Sigma \in \{\Sigma_g, \Sigma_g^1\}$.  The purpose of this paper is to initiate an investigation of the (left) $\Mod(\surface)$-module structure of the Steinberg module.
\begin{equation*}
\St(\surface) = \widetilde{\HH}_{2g-2}(\CC(\surface);\Z)
\end{equation*}
Although Harer has proven that the curve complex $\CC(\surface)$ has the homotopy type of a wedge sum of $2g-2$ spheres, no homologically nontrivial sphere was previously known. 
We show that $\St(\Sigma)$ has a large finite generating set with a generator for each topologically distinct way of gluing the sides of a $4g$-gon to get a surface of genus $g$. We show that the standard identification of sides of the $4g$-gon results is a trivial class in $\St(\Sigma)$  (see Remark~\ref{rem:guess} below), but Corollary~\ref{cor:notriv} gives another identification of sides $\phi_0$ which gives a nontrivial class in $\St(\Sigma)$. This is our first main result

\begin{mainthm}
The homology class $[\phi_0] \in \St(\surface_g^1)$ (resp. $[\phi_0] \in \St(\surface_g)$) is nontrivial for $g \geq 1$.
\end{mainthm}

The number of topologically distinct ways of gluing the sides of a $4g$-gon to get a surface of genus $g$ grows quickly in $g$.  Thus the first finite $\Mod(\Sigma)$-module generating set which we give in  Proposition~\ref{prop:st_pres} below is quite large.  It is therefore somewhat surprising that in Theorem \ref{thm:single_gen} below, we show that this large finite generating set may be reduced to a singleton. This is our second main result.

\begin{mainthm}
\label{mthm:single_gen}
For $g \geq 1$ the Steinberg module $\St(\Sigma)$ is a cyclic $\Mod(\Sigma)$-module generated by the single element described in Section \ref{sec:cyclic}.
\end{mainthm}

Main Theorem~\ref{mthm:single_gen} may be compared to a result of Ash-Rudolph \cite[Theorem~4.1]{AR1979} which implies that the reduced homology of the Tits building for $\SL(n,\Q)$ is a cyclic $\SL(n,\Z)$-module.

%

\subsection{The curve complex and duality for mapping class groups}

We briefly overview the reason that the reduced homology of the curve complex is the dualizing module for any torsion-free, finite index subgroup of $\Mod(\surface)$.  The mapping class group acts properly discontinuously on Teichm\"uller space $\Teich(\surface)$, which is diffeomorphic to $\R^{6g-6+2m}$, where $m$ is the number of marked points on $\surface$.  For any torsion-free, finite index subgroup $\Gamma < \Mod(\surface)$ the quotient of $\Teich(\surface)$ by the properly discontinuous, fixed-point-free action of $\Gamma$ is a manifold. This manifold is not compact but can be compactified either by adding certain ``points at infinity'' following Ivanov \cite{Ivanov1989} or by removing a certain open neighborhood of infinity following Harer \cite{Harer1986}.  In either case, the universal cover $\overline{\Teich}$ of this compactified manifold is called a \emph{bordification} of Teichm\"uller space.  One must then show that $\overline{\Teich}$ is contractible and that $\partial \overline{\Teich}$ has the homotopy type of a wedge of spheres of dimension $2g-2$. It then follows \cite[\S 6.4]{BE1973} that the cohomological dimension of $\Gamma$ is $6g-6+2m-(2g-2)-1$ and the dualizing module of $\Gamma$ is $\widetilde{\HH}_{2g-2}(\partial \overline{\Teich};\Z)$.  Harer establishes these results for surfaces with boundary or punctures in \cite{Harer1986}.  In addition, he shows \cite[Lemma 3.2]{Harer1986} that $\partial \overline{\Teich}$ is $\Mod(\surface)$-equivariantly homotopy equivalent to the curve complex.  Ivanov \cite{Ivanov1987} establishes the same for closed surfaces.

\subsection{The Tits building for $\SL(n,\Q)$}
Throughout this paper we will maintain the viewpoint that the homology of the curve complex should be viewed as an analog for the mapping class group of the homology of the rational Tits building for $\SL(n,\Z)$.  Briefly, the rational Tits building $\TB(n,\Q)$ is the simplicial complex of flags of nontrivial proper subspaces of $\Q^n$ (see \cite[\S V.1 Example 1B]{Brown1998}).  The action of $\SL(n,\Q)$ on $\Q^n$ induces a simplicial action of $\SL(n,\Q)$ on $\TB(n,\Q)$.  For any basis $\mathbf{b}$ for $\Q^n$ the union of all simplices of $\TB(n,\Q)$ whose vertices are subspaces spanned by nonempty proper subsets of $\mathbf{b}$ gives an {\em apartment} of $\TB(n,\Q)$.  Apartments have the homeomorphism type of an $(n-2)$-sphere.

The Steinberg module for $\SL(n,\Z)$ is defined \cite[pg.~437]{BS1973} to be the infinitely generated free abelian group
\begin{equation*}\St(n) = \widetilde{\HH}_{n-2}(\TB(n,\Q);\Z).\end{equation*}
Borel and Serre \cite[Theorem~11.4.2]{BS1973} show that the dualizing module of any torsion-free finite index subgroup of $\SL(n,\Z)$ is $\St(n)$ providing inspiration for Harer's later work \cite{Harer1986} on the mapping class group. 

The Solomon-Tits Theorem ({\it cf.} \cite{Solomon1969}, \cite[\S IV.5 Theorem 2]{Brown1998}) states two things.  Firstly, it says that the Tits building $\TB(n,\Q)$ has the homotopy type of a wedge of infinitely many $(n-2)$-spheres.  Secondly, it says that $\St(n) = \widetilde{\HH}_{n-2}(\TB(n,\Q);\Z)$ is spanned by the homology classes of all apartments.  The action of $\SL(n,\Q)$ is transitive on the set of apartments of $\TB(n,\Q)$ so one sees immediately that $\widetilde{\HH}_{n-2}(\TB(n,\Q);\Z)$ is a cyclic $\SL(n,\Q)$-module.

 In analogy with the first part of the Solomon-Tits Theorem, Harer has shown that the curve complex has the homotopy type of a wedge of $(2g-2)$-spheres.
We prove the analog of the second part in Theorem~\ref{thm:single_gen} below which states that as a $\Mod(\surface)$-module the Steinberg module $\St(\surface)$ is generated by a single element.


Actually, Theorem~\ref{thm:single_gen} more closely resembles a result of Ash-Rudolph \cite[Theorem~4.1]{AR1979} which implies that the reduced homology of the Tits building $\TB(n,\Q)$ for $\SL(n,\Q)$ is a cyclic $\SL(n,\Z)$-module. The action of $\SL(n,\Z)$ is no longer transitive on the set of apartments of $\TB(n,\Q)$, so Ash and Rudolph give a reduction process to rewrite the homology class of the sphere for an arbitrary apartment as a sum of homology classes of spheres of ``integral unimodular'' apartments.  Our analogous reduction process is given in Proposition~\ref{prop:salient} below. 

In Section \ref{sec:notation} below we summarize the results of Harer and others on which or current discussion relys.  In Section \ref{sec:st_pres} we derive a resolution of the Steinberg module which yields a large finite generating set.  Finally, in Section \ref{sec:single_gen} we show that the Steinberg module is generated by a single element.

\medskip
\noindent
{\bf Acknowledgements.} \ 
I owe a huge debt to Juan Souto who pointed out that the best way to find the spheres in the curve complex is to look for balls they bound in Teichm\"uller space.  I am also indebted to Karen Vogtmann for elucidating the work of Harer.  Benson Farb provided immeasurable support and feedback for this project.  I would also like to thank Mladen Bestvina, Matthew Day, Justin Malestein, and Robert Penner for some very helpful discussions.


\section{The Steinberg module, the curve complex and the arc complex}
\label{sec:notation}

Let $\surface_g$ be the surface of genus $g$, and let $\surface_g^1$ be the surface of genus $g$ with $1$ marked point $\ast$.
The {\em mapping class group} $\Mod(\surface_g)$ (resp. $\Mod(\surface_g^1)$) is the group of orientation-preserving self diffeomorphisms of $\surface_g$ (resp. orientation-preserving self diffeomorphisms of $\surface_g^1$ fixing the marked point) modulo diffeomorphisms isotopic to the identity (see \cite{Ivanov2002} for a survey).  An \emph{(essential) curve} in $\surface_g$ is an isotopy class of the image of an embedding of the circle $S^1$ in $\surface_g$ not bounding a disk in $\surface_g$.   An \emph{(essential) curve} in $\surface_g^1$ is an isotopy class of the image of an embedding of $S^1$ in $\surface_g^1 \smallsetminus \{\ast\}$ not bounding a disk or a once-punctured disk in  $\surface_g^1 \smallsetminus \{\ast\}$. Note that for the surface with a marked point, curves may not be isotoped past the marked point.  In either surface, a {\em curve system} is a set of curves (with isotopy class representatives) which can be made to be disjoint.  Since curves are isotopy classes, a curve system cannot have parallel curves. We can partially order curve systems by inclusion.
\begin{definition}[Harvey \cite{Harvey1981}]
\label{def:cc}
For $g \geq 1$ the {\em curve complex} $\CC(\surface_g)$ (resp. $\CC(\surface_g^1)$) for the surface $\surface_g$ (resp. $\surface_g^1$) is the simplicial complex with $n$-simplices corresponding to curve systems with $n+1$ curves and face relation given by inclusion.
\end{definition}
Both $\CC(\surface_g)$ and $\CC(\surface_g^1)$ have the homotopy type of an infinite wedge of spheres of dimension $2g-2$ (see \cite[Theorem~3.5]{Harer1986} and \cite[Theorem~1.4]{IJ2007}).  As stated above, the reduced homology of the curve complex (which is concentrated in dimension $2g-2$) is fundamental to the structure of the mapping class group as a virtual duality group with virtual cohomological dimension $4g-5$ for the closed surface and dimension $4g-3$ for the surface with one marked point \cite[Theorem~4.1]{Harer1986}.
\begin{definition}
\label{def:st_def}
For $g \geq 1$ the {\em Steinberg module} for the mapping class group $\Mod(\surface_g^1)$ is the $\Mod(\surface_g^1)$-module
\begin{equation*}
 \St(\surface_g^1) := \widetilde{\HH}_{2g-2}( \CC(\surface_g^1) ; \Z),
\end{equation*}
and the {\em Steinberg module} for the mapping class group $\Mod(\surface_g)$ is the $\Mod(\surface_g)$-module
\begin{equation*}
 \St(\surface_g) := \widetilde{\HH}_{2g-2}( \CC(\surface_g) ; \Z).
\end{equation*}
\end{definition}

Our aim is to investigate the module structure of the Steinberg modules $\St(\surface_g^1)$ and $\St(\surface_g)$.  A result of Harer cuts our work in half. The mapping class group $\Mod(\surface_g^1)$ acts on the curve complex $\CC(\surface_g)$ by forgetting the marked point.  Hence $\St(\surface_g)$ is naturally a $\Mod(\surface_g^1)$-module.  Harer \cite[Lemma~3.6]{Harer1986} has shown that forgetting the marked point gives a $\Mod(\surface_g^1)$-equivariant homotopy equivalence $\CC(\surface_g^1) \simeq \CC(\surface_g)$.   See the work of Kent, Leininger, and Schleimer in \cite{KLS2009} for more on this projection of curve complexes.  We therefore have the following lemma.
\begin{lemma}[Harer]
\label{lem:same_module}
As $\Mod(\surface_g^1)$-modules $\St(\surface_g^1) \cong \St(\surface_g)$.
\end{lemma}
This in turn allows an immediate conclusion about the $\Mod(\surface_g^1)$-module structure of $\St(\surface_g^1)$ which is not at all apparent from direct observation of the action of $\Mod(\surface_g^1)$ on $\CC(\surface_g^1)$.
\begin{corollary} 
\label{cor:factors}
The action of $\Mod(\surface_g^1)$ on $\St(\surface_g^1)$ factors through its quotient $\Mod(\surface_g)$.
\end{corollary}

In light of Lemma~\ref{lem:same_module} we will focus exclusively on the $\Mod(\surface_g^1)$-module structure of $\St(\surface_g^1)$.
Instead of calculating the homology of the curve complex $\CC(\surface_g^1)$ directly, we will work in the {\em arc complex} for $\surface_g^1$ (see Definition~\ref{def:ac} below).  Lemma~\ref{lem:same_module} is especially fortunate since no structure analogous to the arc complex is readily available for the surface $\surface_g$.

An {\em (essential) arc} in $\surface_g^1$ is an isotopy class of the image of an embedded loop based at the marked point $\ast$ which does not bound a disk in $\surface_g^1$.
An {\em arc system} is a set of arcs (with isotopy class representatives) which intersect only at the marked point. As with curve systems, since arc arc systems are sets of isotopy classes, arc systems may not have parallel arcs.  Henceforth we will make no distinction between an arc system and a set of representatives of each arc intersecting only at the marked point.
\begin{definition}[Harer]
\label{def:ac}
The {\em arc complex} $\AC = \AC(\surface_g^1)$ is the simplicial complex with $n$-simplices corresponding to arc systems with $n+1$ arcs and face relation given by inclusion.
\end{definition}
An arc system $\alpha = \{\alpha_0, \cdots, \alpha_n \}$ {\em fills} $\surface_g^1$ if the connected components of $\surface_g^1 \smallsetminus \bigcup \alpha$ are all disks.  The minimum number of arcs needed to fill $\surface_g^1$ is $2g$.  In what follows we will want to keep careful track of the number of arcs in a filling system, so we will say that a filling arc system \emph{$k$-fills} $\surface_g^1$ if the arc system has $2g+k$ arcs.  Notice that a $k$-filling system cuts the surface into $k+1$ disks.
\begin{definition}[Harer]
The {\em arc complex at infinity} $\AC_\infty = \AC_\infty(\surface_g^1)$ is the simplicial subcomplex of $\AC(\surface_g^1)$ which is the union of the simplices of $\AC(\surface_g^1)$ whose vertex sets do not fill $\surface_g^1$.
\end{definition}
Any arc system with fewer than $2g$ arcs cannot fill $\surface_g^1$ so $\AC_\infty(\surface_g^1)$ contains the entire $(2g-2)$-skeleton of $\AC(\surface_g^1)$. 

The name ``arc complex at infinity'' calls for some explanation. A very nice discussion is given in \cite[\S 2]{Harer1988}.  The idea is this.  Using a fundamental construction of Jenkins and Strebel \cite{Strebel1984} it is possible to associate to each point in Teichm\"uller space $\Teich (\surface_g^1)$ (the space of finite area marked complete hyperbolic metrics on the surface $\surface_g^1 \smallsetminus \{\ast\}$) an embedded metric graph in $\surface_g^1 \smallsetminus \{\ast\}$.  One can then homeomorphically identify $\Teich (\surface_g^1)$ with $ \AC(\surface_g^1) \smallsetminus \AC_\infty(\surface_g^1)$ using the arc system in $\surface_g^1$ dual to this graph (see \cite[\S 2]{Harer1988} for details).  Thus one may think of $\AC_\infty(\surface_g^1)$ as extra ``points at infinity'' attached to Teichm\"uller space.

Harer \cite[Theorem~3.4]{Harer1986} defined a continuous map $ \Psi: \AC_\infty (\surface_g^1) \to \CC (\surface_g^1)$ and showed that it is a homotopy equivalence (see \S \ref{sec:curv_sphere} below for more on this map).  Hence, from Definition~\ref{def:st_def} we get another characterization of the Steinberg module
\begin{equation}
\label{eq:inf_def}
 \St(\surface_g^1) \cong \widetilde{\HH}_{2g-2}( \AC_\infty(\surface_g^1) ; \Z).
\end{equation}


\section{A resolution of the Steinberg module}
\label{sec:st_pres}

In this section we will give a resolution of the Steinberg module $\St(\surface_g^1)$ as a $\Mod(\surface_g^1)$-module.
The approach here is analogous to Ash's simplification \cite[\S 1]{Ash1994} of Lee and Szczarba's resolution of the Steinberg module for $\SL(n,\Z)$ given in \cite[\S 3]{LS1976} (see \cite{Gunnells2000} for a nice summary).

We will need the following two results of Harer.
\begin{theorem}[Harer]
\label{thm:har_cont}
 The arc complex $\AC (\surface_g^1)$ is contractible.
\end{theorem}
\begin{theorem}[Harer] 
\label{thm:har_spheres}
The arc complex at infinity $\AC_\infty(\surface_g^1)$ is homotopy equivalent to a wedge of spheres of dimension $2g-2$.
\end{theorem}
Theorem~\ref{thm:har_cont} was established in \cite[Theorem 1.5]{Harer1985} (see \cite{Hatcher1991} for a concise proof) and 
Theorem~\ref{thm:har_spheres} was shown in \cite[Theorem 3.3]{Harer1986}.  An alternate proof is provided in \cite[Theorem 6.6]{Ivanov1987}

Consider this portion of the long exact sequence of reduced homology groups for the pair of spaces $( \AC, \AC_\infty)$.
\begin{equation}
 \HH_{k+1}(\AC;\Z) \to \HH_{k+1}(\AC / \AC_\infty;\Z) \to \widetilde{\HH}_{k}(\AC_\infty;\Z) \to \widetilde{\HH}_{k}(\AC;\Z).
\end{equation}
By Theorem~\ref{thm:har_cont} the first and last groups in this sequence are trivial for $k \geq 0$; consequently,
\begin{equation}
\label{eq:a_infinity}
 \HH_{k+1}(\AC / \AC_\infty;\Z) \cong \widetilde{\HH}_{k}(\AC_\infty;\Z) \ \ \ \  \mbox{for $k \geq 0$}.
\end{equation}
Now combining equations (\ref{eq:inf_def}) and (\ref{eq:a_infinity}) we arrive at a very useful description of the Steinberg module,
\begin{equation}
\label{eq:st_working_def}
 \St(\surface_g^1) \cong \HH_{2g-1}(\AC / \AC_\infty;\Z).
\end{equation}

The chain complex $(\mathrm{C}_\ast(\AC / \AC_\infty;\Z), \partial)$ for cellular homology of the space $\AC / \AC_\infty$ will provide us with a resolution for $\St(\surface_g^1)$ as a $\Mod(\surface_g^1)$-module.  We define the chain complex  $(\fsys_\ast,\partial)$ to be the shifted complex with
\begin{equation}
\label{eq:chain_def}
 \fsys_k := \mathrm{C}_{2g - 1 + k}( \AC / \AC_\infty;\Z)
\end{equation}
and the same boundary maps as $(\mathrm{C}_\ast(\AC / \AC_\infty;\Z), \partial)$.
\begin{proposition}
\label{prop:stein_res}
Let $(\fsys_\ast,\partial)$ be the chain complex defined in Equation~(\ref{eq:chain_def}).  The exact sequence
\begin{equation*}
 0 \to \fsys_{4g-3} \stackrel{\partial}{\to} \cdots \stackrel{\partial}{\to} \fsys_2 \stackrel{\partial}{\to} \fsys_1 \stackrel{\partial}{\to} \fsys_0 \to \St(\surface_g^1) \to 0
\end{equation*}
is a finite $\Mod(\surface_g^1)$-module resolution\footnote{Compare with the resolution of the Steinberg module for $\SL(n,\Z)$ given in \cite[\S 1]{Ash1994}.} for the Steinberg module $\St(\surface_g^1)$.
\end{proposition}

\begin{proof}
By Theorem~\ref{thm:har_spheres} the arc complex at infinity $\AC_\infty$ has the homotopy type of a wedge of $(2g-2)$-dimensional spheres.  Therefore
\begin{equation*}
\begin{split}
\HH_{k}(\fsys_\ast) & = \HH_{2g-1+k}(\AC / \AC_\infty;\Z) \\
 & \cong \widetilde{\HH}_{2g-2+k}( \AC_\infty;\Z) \\
 & = \left \{\begin{array}{ll} 0, & k > 0 \\ \St(\surface_g^1), & k=0.\end{array} \right.
\end{split}
\end{equation*}
In other words, the chain complex $(\fsys_*,\partial)$ gives a resolution of $\St(\surface_g^1)$.

The maximum number of arcs in an arc system occurs when the arcs comprise the one-skeleton of a one-vertex triangulation of $\surface_g^1$.  Using euler characteristic one may then conclude that an arc system has at most $6g-3$ arcs so $\AC$ and hence $\AC / \AC_\infty$ is $6g-4$ dimensional.   It follows that $\fsys_k = \mathrm{C}_{2g - 1 + k}( \AC / \AC_\infty;\Z) = 0$ when $k > 6g-4 - 2g +1 = 4g-3$.
 \end{proof}

A $k$-filling system corresponds to a unique $(2g + k-1)$-cell in $\AC/\AC_\infty$ with the cell decomposition inherited from $\AC$.  When the orientation of this $(2g + k-1)$-cell is important we will specify it via an orientation on the unique $(2g + k-1)$-simplex in $\AC$ mapping to it.  An \emph{oriented} $k$-filling system will be a $k$-filling system together with an order on the set of arcs in the system up to alternating permutations.  By definition $\fsys_k = \mathrm{C}_{2g - 1 + k}( \AC / \AC_\infty;\Z)$ is the set of finite $\Z$-linear combinations of oriented $k$-filling systems\footnote{Translational Note: For the mapping class group, $0$-filling systems play a similar role to that of \emph{modular symbols} for $\SL(n,\Z)$.}.

There are a finite number of topologically distinct ways to glue the sides of polygons to get a one-vertex cell decomposition of a surface of genus $g$, and by the usual change of coordinates principle $\Mod(\surface_g^1)$ is transitive on the set of cell decompositions of the same topological type. Consequently, as a $\Mod(\surface_g^1)$-module, $\fsys_k$ is generated by a finite number of oriented $k$-filling systems.

Note that in general $\fsys_k$ is not quite a free $\Mod(\surface_g^1)$-module since some filling arc systems have nontrivial (but always finite cyclic) stabilizers in $\Mod(\surface_g^1)$.

There are two modifications either of which makes the resolution in Proposition~\ref{prop:stein_res} {\em projective}. Firstly, if coefficients are taken in $\Q\Mod(\surface_g^1)$ instead of $\Z\Mod(\surface_g^1)$ then Proposition~\ref{prop:stein_res} does give a projective resolution of the rational homology of the curve complex (see \cite[pg. 30 Exercise 4]{Brown}).  Alternatively, if one restricts to some torsion-free subgroup $\Gamma < \Mod(\surface_g^1)$ then Proposition~\ref{prop:stein_res} gives a free $\Gamma$-module resolution of $\St(\surface_g^1)$.

It is worth pointing out that in the cases where the $\Gamma$-module resolution of $\St(\surface_g^1)$ in Proposition \ref{prop:stein_res} is projective, we get the following formula for computing group cohomology of $\Gamma$.
\begin{proposition}
\label{prop:cohom_calc}
If $\Gamma < \Mod(\surface_g^1)$ is torsion free and finite-index subgroup then for any $\Gamma$-module $A$
$$ \HH^n(\Gamma;A) \cong \HH_{4g-3-n}(\fsys_\ast \otimes_\Gamma A)$$
\end{proposition}
\begin{proof}
 The proof is immediate from duality and properties of $\Tor$.
By duality
$$\HH^n(\Gamma;A) \cong \HH_{4g-3-n}(\Gamma;\St \otimes_\Z A)$$
where $\St \otimes_\Z A$ has $\Gamma$-module structure given by the diagonal action: $\gamma (d \otimes a) = \gamma d \otimes \gamma a$.
$$\HH_{4g-3-n}(\Gamma;\St \otimes_\Z A) \cong \Tor_{4g-3-n}^\Gamma(\St,A) \cong \HH_{4g-3-n}(\fsys_\ast \otimes_\Gamma A).$$
See \cite[III.2]{Brown}.
\end{proof}

\begin{corollary}
If $\Gamma < \Mod(\surface_g^1)$ is a torsion free and finite-index subgroup then the resolution in Proposition \ref{prop:stein_res} is the shortest possible projective $\Gamma$-module resolution of $\St(\surface_g^1)$.
\end{corollary}

For the purposes of this paper we will make do with the resolution as is.  Note that the tail of this resolution almost provides us with a $\Mod(\surface_g^1)$-module presentation for $\St(\surface_g^1)$.  It is not exactly a presentation because $\fsys_0$ is not a free $\Mod(\surface_g^1)$-module.  However, it can easily be augmented to give a $\Mod(\surface_g^1)$-module presentation by artificially adding stabilizer relations which account for the accidental symmetries of certain $0$-filling systems.

\begin{proposition}
\label{prop:st_pres}
Let $g \geq 1$.  Choose a set of oriented representatives $G = \{ \phi_0, \cdots, \phi_n \}$ of each $\Mod(\surface_g^1)$-orbit of the set of $0$-filling systems.  For each $\phi_i \in G$ let $h_i \in \Mod(\surface_g^1)$ be a generator of the stabilizer of the arc system for $\phi_i$, and let $e_i = \pm 1$ be the sign of the permutation that the mapping class $h_i$ induces on this set of arcs.
Also choose a set of oriented representatives $\{ \rho_0, \cdots, \rho_m \}$ of each $\Mod(\surface_g^1)$-orbit of the set of $1$-filling systems. $\St(\surface_g^1)$ has a presentation
\begin{equation*}
 \St(\surface_g^1) \cong \big\langle \phi_0, \cdots, \phi_n  \ \big| \ \partial \rho_0, \cdots, \partial \rho_m,  (1 - e_0h_0)\phi_0 , \cdots, (1 - e_nh_n)\phi_n \big\rangle.
\end{equation*}
Forgetting the marked point gives a surjective ring homomorphism
\begin{equation*}
\Z \Mod(\surface_g^1) \surjection \Z \Mod(\surface_g).
\end{equation*}
If coefficients in the above presentation are sent to their images under this ring homomorphism then we get a $\Mod(\surface_g)$-module presentation
\begin{equation*}
 \St(\surface_g) \cong \big\langle \phi_0, \cdots, \phi_n  \ \big| \ \partial \rho_0, \cdots, \partial \rho_m,  (1 - e_0h_0)\phi_0 , \cdots, (1 - e_nh_n)\phi_n \big\rangle.
\end{equation*}
\end{proposition}

\begin{proof}
By Proposition~\ref{prop:stein_res} we have the $\Mod(\surface_g^1)$-module isomorphism
\begin{equation*}
 \St(\surface_g^1) \cong \frac{\fsys_0}{ \partial \fsys_1}.
\end{equation*}
By definition $\fsys_0$ is spanned by oriented $0$-filling systems as a $\Z$-module.  Every oriented $0$-filling system is of the form $\pm h \phi_i$ for some $\phi_i \in G$ and $h \in \Mod(\surface_g^1)$.  Hence $G$ spans $\fsys_0$ as a $\Mod(\surface_g^1)$-module.  The only $\Z\Mod(\surface_g^1)$-linear dependencies in the set $G$ arise from stabilizers of $0$-filling systems.  The arc system for a $0$-filling system cuts the surface into a single $4g$-gon, so its stabilizer must be a subgroup of the rotational symmetries of the regular $4g$-gon. Hence we have the presentation
\begin{equation}
\label{eq:fill0_pres}
 \fsys_0 \cong \big\langle \phi_0, \cdots, \phi_n  \ \big| \ (1 - e_0h_0)\phi_0 , \cdots, (1 - e_nh_n)\phi_n \big\rangle.
\end{equation}
and the presentation for $\St(\surface_g^1)$ then follows.

For the closed surface $\surface_g$ we do not have a resolution for $\St(\surface_g)$ as a $\Mod(\surface_g)$-module; however, we do get a $\Mod(\surface_g)$-module presentation for $\St(\surface_g)$ by taking co-invariants (defined below) in the presentation for $\St(\surface_g^1)$.  Let $P < \Mod(\surface_g^1)$ be the \emph{point pushing subgroup} of $\Mod(\surface_g^1)$; that is, the kernel of $\mu$ in the Birman Exact Sequence \cite{Birman1969}
\begin{equation}
\label{eq:bes}
1 \to \pi_1(\surface_g) \to \Mod(\surface_g^1) \stackrel{\mu}{\to} \Mod(\surface_g) \to 1 \ \ \ \ \ \mbox{(for $g \geq 2$)}
\end{equation}
or for $g=1$  the trivial kernel of $\mu$ in the exact sequence
\begin{equation}
\label{eq:g1bes}
 1 \to \Mod(\surface_1^1) \stackrel{\mu}{\cong} \Mod(\surface_1) \to 1
\end{equation}
The $P$-co-invariants $M_P$ of a $\Mod(\surface_g^1)$-module $M$ are the quotient of $M$ by the $P$-submodule generated by the set $\{pm-m|p \in P, m \in M\}$. The $P$-co-invariants $M_P$ have a $\Mod(\surface_g)$-module structure \cite[\S II.2 Problem~3]{Brown}.

 The exact sequence
\begin{equation*} \fsys_1 \to \fsys_0 \to \St(\surface_g^1) \to 0\end{equation*}
remains exact \cite[\S II.2]{Brown} after taking $P$-co-invariants to get
\begin{equation*} (\fsys_1)_P \to (\fsys_0)_P \to \St(\surface_g^1)_P \to 0. \end{equation*}
By Corollary~\ref{cor:factors} and Lemma~\ref{lem:same_module} we have $\Mod(\surface_g)$-module isomorphisms
\begin{equation*}\St(\surface_g^1)_P \cong \St(\surface_g^1) \cong \St(\surface_g).\end{equation*}
Hence, as a $\Mod(\surface_g)$-module $\St(\surface_g)$ satisfies
\begin{equation*}
 \St(\surface_g) \cong \frac{(\fsys_0)_P}{ \partial (\fsys_1)_P}.
\end{equation*}
All that is left is to observe that taking the $P$-co-invariants of $\fsys_0$ achieves the same effect as sending the coefficients in the presentation in (\ref{eq:fill0_pres}) to their images in $\Z \Mod(\surface_g)$.
 \end{proof}


\section{The Steinberg module is cyclic}
\label{sec:cyclic}

Part one of the Solomon-Tits Theorem states that the Tits building $\TB(n,\Q)$ for $\SL(n,\Q)$ has the homotopy type of a wedge of spheres of dimension $n-2$.  The second part says that  the $\SL(n,\Q)$-module $\St(n) = \widetilde{\HH}_{n-2}(\TB(n,\Q);\Z)$ is a cyclic $\SL(n,\Q)$-module generated by the $(n-2)$-sphere coming from a single apartment.  In analogy with the first part of the Solomon-Tits Theorem, Harer has shown that the curve complex has the homotopy type of a wedge of $(2g-2)$-spheres.
In this section we will prove the analog of the second part of the Solomon-Tits Theorem in Theorem~\ref{thm:single_gen} below which states that as a $\Mod(\surface)$-module the Steinberg module $\St(\surface)$ is generated by a single element.

In fact, our Theorem~\ref{thm:single_gen} more closely resembles a result of Ash-Rudolph \cite[Theorem~4.1]{AR1979} which implies that the reduced homology of the Tits building $\TB(n,\Q)$ for $\SL(n,\Q)$ is a cyclic $\SL(n,\Z)$-module. The action of $\SL(n,\Z)$ is no longer transitive on the set of apartments of $\TB(n,\Q)$, so one must rely on a reduction process to rewrite the homology class of the sphere for an arbitrary apartment as a sum of homology classes of spheres of ``integral unimodular'' apartments.

The arc complex at infinity does not come with any apartment structure, but  $0$-filling systems come in a finite number of types based on their $\Mod(\surface_g^1)$-orbits.  Proposition~\ref{prop:salient} below gives a reduction algorithm to write the class of any oriented $0$-filling system in the Steinberg module for the mapping class group as a linear combination of classes of oriented $0$-filling systems of a single type.

In order to state and prove Theorem~\ref{thm:single_gen} we will first discuss a notational convenience which will allow us to keep track of complicated filling arc systems.


\subsection{Chord diagrams}
\label{sec:chord_diagrams}

Already in genus two it is difficult to keep track of filling arc systems and their symmetries when drawn on surfaces.  Chord diagrams provide a convenient notation for this. (See \cite[Figure~9]{Mosher1994} and the related discussion for an introduction.)

An (unlabelled) \emph{chord diagram} (see Figure~\ref{fig:fat_graph}, left) will be a regular $2n$-gon (which will always be depicted as a circle) with vertices paired off so that no two adjacent vertices are paired.  The pairing will be indicated by $n$ \emph{chords} in the chord diagram.  We will call an edge of the $2n$-gon an \emph{outer edge} of the chord diagram.  A \emph{labelled chord diagram} (see Figure~\ref{fig:chord_diagram}, right) will be a chord diagram in which each chord is labelled on one side by an element of $\pi_1 (\surface_g,\ast)$ and on the other side by its inverse.  In practice we will only label one side of each chord with the assumption that the other side is labelled with the inverse element in $\pi_1 (\surface_g,\ast)$. Two labelled or unlabelled chord diagrams are the same if they can be made to agree after rotation of the $2n$-gon.  A labelling of a chord diagram will be called \emph{proper} if it comes from a filling arc system in the surface $\surface_g^1$.  We implicitly require all labellings of chord diagrams to be proper.

In an arc system on the surface $\surface_g^1$, arcs leave and return to a neighborhood of the marked point $\ast$ in a certain cyclic order.  To each filling arc system in $\surface_g^1$ with $n$ arcs, associate the labelled chord diagram obtained by placing a $2n$-gon in a neighborhood of the marked point so that each arc enters and leaves the $2n$-gon at a vertex.  For each arc in the arc system connect the two vertices of the $2n$-gon on that arc with a chord.  The two sides of the chord correspond to the two sides of the arc.  Label each side of the chord by the element of $\pi_1(\surface_g,\ast)$ obtained by following the arc on the corresponding side in a counter-clockwise direction (see Figure~\ref{fig:chord_diagram}). 

\begin{figure}[!ht]
\begin{center}
\resizebox{0.75\textwidth}{!}{%
\begin{minipage}[c]{2.8in}
\begin{center}
\includegraphics[scale=0.5]{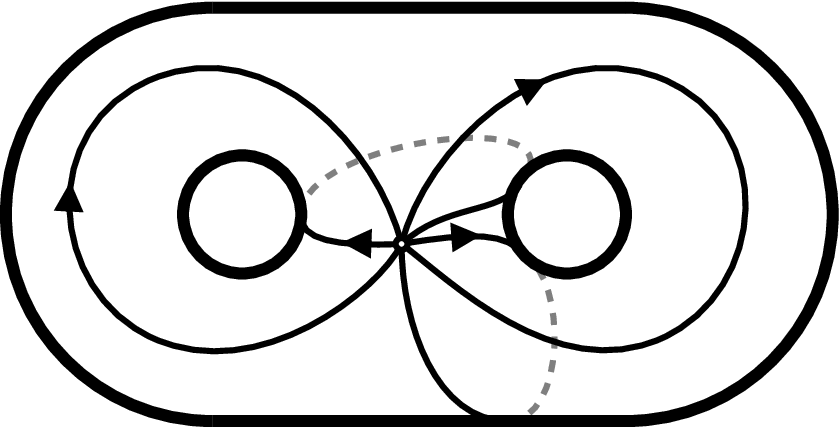}
\put(-108,53){$\ast$}
\put(-195,55){$x$}
\put(-127,37){$y$}
\put(-70,89){$z$}
\put(-111,15){$w$}
\end{center}
\label{fig:g2_fill}
\end{minipage}
\hspace{0.2in}
\begin{minipage}[c]{1.4in}
\begin{center}
\includegraphics[scale=1.0]{chords-example.epsi}
\put(-90,29){$x^{-1}$}
\put(-70,25){$x$}
\put(-65,60){$y^{-1}$}
\put(-58,48){$y$}
\put(-45,78){$z^{-1}$}
\put(-56,78){$z$}
\put(-54,20){$w$}
\put(-45,12){$w^{-1}$}
\end{center}
\end{minipage}
}
\end{center}
\caption{\label{fig:chord_diagram}A filling arc system (left) and the corresponding labelled chord diagram (right).}
\end{figure}

We will do most of our calculations using chord diagrams, so it will be convenient to be able to recover certain characteristics of the surface and embedded filling arc system corresponding to a given chord diagram.  Firstly we briefly explain how to reconstruct the surface.  From an $n$-chord diagram we may construct a surface with a certain embedded, one-vertex graph as follows.  A \emph{fat graph} is a regular neighborhood of a graph embedded in a surface together with the embedding of the graph, or equivalently, a graph together with, for each vertex of the graph, a cyclic order on the termini of the edges incident with that vertex. A chord diagram specifies a one-vertex fat graph with one edge for each chord, where the cyclic order on the edge termini corresponds to the cyclic order on endpoints of the chords in the diagram (see Figure~\ref{fig:fat_graph}).  
\begin{figure}[!ht]
\begin{center}
\resizebox{0.50\textwidth}{!}{%
\begin{minipage}[c]{1.4in}
\begin{center}
\includegraphics[scale=1.0]{fat-graph-1.epsi}
\end{center}
\end{minipage}
\hspace{.07\linewidth}
\begin{minipage}[c]{1.1in}
\begin{center}
\includegraphics[scale=0.7]{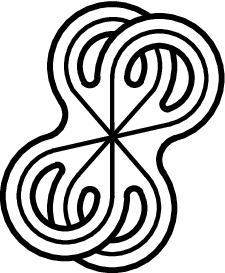}
\end{center}
\end{minipage}
}
\end{center}
\caption{\label{fig:fat_graph}A chord diagram (left) and the corresponding fat graph (right).}
\end{figure}

Now we may glue disks to each of the $b$ boundary components of the fat graph to get a closed surface in which the edges of the fat graph form a filling arc system. Note that one generally considers the fat graph dual to an arc system \cite{PM2007} \cite{ABP2007}, but here our arc system and fat graph agree.  It will be useful to know the genus $g$ of this closed surface, which is easily calculated using the number $n$ of edges in the fat graph and the number $b$ of boundary components of the fat graph, to be
\begin{equation}
\label{eq:euler}
 g = \frac{n+1-b}{2}.
\end{equation}
In particular, this shows that a chord diagram with $2g$ chords corresponds to a $0$-filling system in a surface of genus $g$ precisely when the corresponding fat graph has one boundary component.  

In light of Equation~(\ref{eq:euler}) above we would like to be able to quickly read off the number $b$ of boundary components in the fat graph corresponding to a given chord diagram.  In fact it is easy to see that $b$ is the number of \emph{cycles} in the chord diagram.  A {\em cycle} in a chord diagram (see Figure~\ref{fig:cycles}) is an alternating sequence of chords and outer edges of the chord diagram obtained by starting at a point just inside an outer edge of the diagram and walking along in a clockwise direction keeping the outer edge on one's left until a chord is encountered, turning right and then following the chord keeping it on the left until an outer edge is encountered, turning right and repeating until one returns to the starting point. 
Notice that for each cycle in a labelled chord diagram the products of the labels in the  cycle must be $1 \in \pi_1(\surface_g)$, since the corresponding loop in the surface bounds a disk.

\begin{figure}[h]
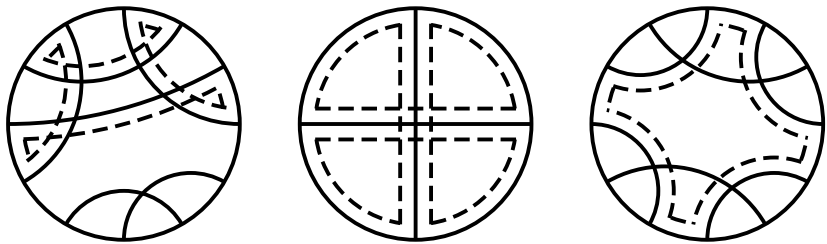

\begin{center}
\resizebox{0.75\textwidth}{!}{%
\begin{minipage}[c]{1.3in}
\begin{center}
\includegraphics[scale=1.0]{cycle-1.epsi}
\end{center}
\end{minipage}
\hspace{.02\linewidth}
\begin{minipage}[c]{1.3in}
\begin{center}
\includegraphics[scale=1.0]{cycle-2.epsi}
\end{center}
\end{minipage}
\hspace{.02\linewidth}
\begin{minipage}[c]{1.3in}
\begin{center}
\includegraphics[scale=1.0]{cycle-3.epsi}
\end{center}
\end{minipage}
}
\end{center}
\caption{\label{fig:cycles}Cycles in chord diagrams.}
\end{figure}

Two chords are \emph{parallel} if along with two outer edges they bound a rectangular cycle.  For example, the left chord diagram in Figure~\ref{fig:cycles} has two parallel chords. Since we will only be concerned with arc systems with no parallel arcs, we will not consider chord diagrams with parallel chords.
We may identify $k$-filling systems for the surface $\surface_g^1$ with (properly) labelled chord diagrams with $2g+k$ chords, $k + 1$ cycles and no parallel chords. 


As an oriented $k$-filling system gives an oriented cell in $\AC / \AC_\infty$, we may take its boundary in the chain complex $\fsys_\ast$ by taking the boundary of the corresponding oriented simplex in $\AC$ and projecting that boundary back to $\AC / \AC_\infty$.  The boundary of a simplex is an alternating sum of the codimension-1 faces.  Thus the boundary of a $k$-filling system $\alpha$ will be a linear combination of all the $(k-1)$-filling systems obtained from $\alpha$ by removing one arc.  Removing certain arcs from $\alpha$ may leave an arc system which no longer fills $\surface_g^1$. The simplex of $\AC$ for such an arc system is contained in $\AC_\infty$ and hence is trivial in $\fsys_{k-1}$.  In the language of chord diagrams we will be able to tell when this has happened by counting cycles and applying Equation~(\ref{eq:euler}).

The mapping class group $\Mod(\surface_g^1)$ acts on the fundamental group of the closed surface of genus $g$.  This gives a left action of $\Mod(\surface_g^1)$ on labelled chord diagrams by modifying the labels.  This action of $\Mod(\surface_g^1)$ preserves the underlying unlabelled chord diagram for a $k$-filling system.  Conversely, by the usual change of coordinates principle, any homeomorphism of fat graphs extends to a homeomorphism of the surface so the mapping class group is transitive on the set of proper labellings for a fixed unlabelled chord diagram.  Thus the $\Mod(\surface_g^1)$-orbits of $k$-filling systems correspond precisely to unlabelled chord diagrams with $2g+k$ chords and $k + 1$ cycles.  
 
\begin{example}
 Figure~\ref{fig:gen_2} depicts the $4$ orbits of the action of $\Mod(\surface_2^1)$ on the set of $0$-filling systems for the surface, $\surface_2^1$.  These are all $4$-chord diagrams with one cycle and (redundantly) no parallel chords.
\end{example}
\begin{figure}[!ht]
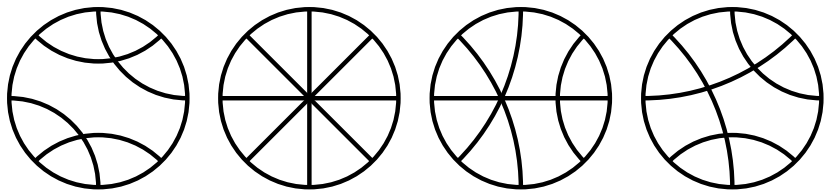

\begin{center}
\resizebox{0.75\textwidth}{!}{%
\begin{minipage}[c]{0.9in}
\begin{center}
\includegraphics[scale=1.0]{g2-1.epsi}
\end{center}
\end{minipage}
\begin{minipage}[c]{0.9in}
\begin{center}
\includegraphics[scale=1.0]{g2-2.epsi}
\end{center}
\end{minipage}
\begin{minipage}[c]{0.9in}
\begin{center}
\includegraphics[scale=1.0]{g2-3.epsi}
\end{center}
\end{minipage}
\begin{minipage}[c]{0.9in}
\begin{center}
\includegraphics[scale=1.0]{g2-4.epsi}
\end{center}
\end{minipage}
}
\end{center}
\caption{\label{fig:gen_2}The unlabelled $4$-chord diagrams corresponding to the four $\Mod(\surface_g^1)$-orbits of $0$-filling systems in the genus $2$ surface with one marked point.}
\end{figure}


\subsection{A singleton generating set for the Steinberg module}
\label{sec:single_gen}

As shown in \S \ref{sec:st_pres} we have the $\Mod(\surface_g^1)$-module isomorphism from Proposition~\ref{prop:st_pres}
\begin{equation*}\St(\surface_g^1) \cong \frac{\fsys_0}{ \partial \fsys_1}.\end{equation*}
Although $\fsys_0$ and $\fsys_1$ are finitely generated $\Mod(\surface_g^1)$-modules, the number of $\Mod(\surface_g^1)$-orbits of $0$-filling systems grows very quickly as a function of $g$ (see \cite{HZ1986}).  It is therefore somewhat surprising that the Steinberg module is generated by the class of a single $0$-filling system.
\begin{theorem}[The Steinberg module is cyclic]
\label{thm:single_gen}
Let $g \geq 1$, and let $\phi_0 \in \fsys_0$ be the $0$-filling system given in Figure~\ref{fig:st_gen}.  Let $[\phi_0]$ be the class of $\phi_0$ in $\St(\surface_g^1) = \fsys_0 / \partial \fsys_1$.  Then $\St(\surface_g^1)$ is generated as a $\Mod(\surface_g^1)$-module by $[\phi_0]$.  For the closed surface let $[\phi_0]$ be the class of $\phi_0$ in $\St(\surface_g) = (\fsys_0)_P / \partial (\fsys_1)_P$ ({\em cf.} proof of Proposition~\ref{prop:st_pres}).  Similarly, $\St(\surface_g)$ is generated as a $\Mod(\surface_g)$-module by $[\phi_0]$.
\end{theorem}

Theorem~\ref{thm:single_gen} will be proved using Propositions~\ref{prop:salient} and \ref{prop:connected} below.

\begin{corollary}
\label{cor:notriv}
$[\phi_0] \in \St(\surface_g^1)$ (resp. $[\phi_0] \in \St(\surface_g)$) is nontrivial for $g \geq 1$.
\end{corollary}
\begin{proof}[Proof of Corollary~\ref{cor:notriv}]
If $[\phi_0]$ were trivial then by Theorem~\ref{thm:single_gen} the Steinberg module $\St(\surface_g^1)$ would be trivial.  But then $\Mod(\surface_g^1)$ would have a finite index subgroup $\Gamma$ whose dualizing module is trivial.  This would imply that $\Gamma$ must be the trivial group and hence that $\Mod(\surface_g^1)$ is finite (see \cite{IJ2007} for more).
 \end{proof}

\begin{figure}[!ht]
\begin{center}
\includegraphics[angle=100]{generator.epsi}
\put(-62,82){$x_1$}
\put(-52,66){$x_2$}
\put(-44,46){$x_3$}
\put(-50,28){$x_4$}
\put(-80,82){$x_{2g}$}
\put(-92,66){$x_{2g-1}$}
\put(-99,46){$x_{2g-2}$}
\put(-92,28){$x_{2g-3}$}
\put(-79,0){\huge$\cdots$}
\put(-71,-24){\large $\phi_0$}
\end{center}
\caption{The generator for $\St(\surface_g^1)$\label{fig:st_gen}}
\end{figure}
\begin{remark}
\label{rem:guess}
A natural first guess at a single generator for $\St(\surface_g^1)$ is the arc system coming from the standard identification of the the $4g$-gon which corresponds to the labelled chord diagram in Figure~\ref{fig:guess}, but in fact by Proposition~\ref{prop:connected} below, the class of this $0$-filling system is trivial in $\St(\surface_g^1)$.
\end{remark}
\begin{figure}[!ht]
\begin{center}
\includegraphics[angle=100]{guess.epsi}
\put(-62,82){$y_1$}
\put(-44,66){$y_2$}
\put(-36,44){$y_3$}
\put(-48,23){$y_4$}
\put(-80,82){$y_{2g}$}
\put(-99,64){$y_{2g-1}$}
\put(-106,42){$y_{2g-2}$}
\put(-92,22){$y_{2g-3}$}
\put(-79,0){\huge$\cdots$}
\end{center}
\caption{A trivial class in $\St(\surface_g^1)$\label{fig:guess}}
\end{figure}

A filling arc system describes a decomposition of the surface $\surface_g^1$ into polygons.  If any of the polygons has more than $3$ vertices, one may add an arc connecting two non-adjacent vertices of that polygon to get a filling arc system with one more arc.  This same process can be done from the point of view of chord diagrams.  Given a $k$-filling system $\alpha$ in $\surface_g^1$ represented as a labelled chord diagram, the outer edges of $\alpha$ are partitioned into $k+1$ disjoint sets (corresponding to the polygons mentioned above) according to which of the $k+1$ cycles they belong to.  A $(k+1)$-filling system can be created by adding a new chord (not parallel to any of the chords of $\alpha$) connecting any two outer edges of $\alpha$ in the same cycle.  Note that since a $0$-filling system has a single cycle, \emph{any} new (non-parallel) chord gives a $1$-filling system.

Two chords in a chord diagram \emph{cross} if the cyclic order on their endpoints forces them to.  The finest partition of the set of chords in which each pair of crossing chords is in the same set will give the \emph{connected components} of the chord diagram.  For example the chord diagram in Figure~\ref{fig:st_gen} is connected, while the chord diagram in Figure~\ref{fig:guess} has $g$ connected components and so is disconnected for $g >1$.   
A cycle {\em traverses} a chord if it follows along the chord.  Note that a cycle can traverse a chord 0, 1, or 2 times.  Given a cycle $\sigma$ and a chord $c$ on a chord diagram $\alpha$ we will say that{ \em $\sigma$  remains on one side of $c$} if the cycle $\sigma$ never traverses any chords that cross $c$.  Each chord $c$ in a chord diagram cuts the chord diagram into two sets whose closures are disks.  Notice that the cycle $\sigma$ remains on one side of $c$ precisely when $\sigma$ is contained in one of those disks.  Finally note that each chord $c$ is either traversed twice by a single cycle which must also traverse a chord which crosses $c$ or is traversed once by two different cycles.

\begin{proposition}
\label{prop:connected}
 If $\alpha$ is a disconnected $0$-filling system then $[\alpha]$ is trivial in $\St(\surface_g^1) = \fsys_0 / \partial \fsys_1$.
\end{proposition}

\begin{proof}
Add a chord $c$ to $\alpha$ which does not cross any of the chords of $\alpha$ and such there are nonempty connected components of $\alpha$ on both sides of $c$.  Let $\alpha_c$ be the resulting chord diagram.  Firstly, we claim that $\alpha_c$ is a $1$-filling system.  This will be the case unless the chord $c$ is parallel to some chord $c'$ of $\alpha$.  Suppose that such a chord $c'$ exists.  Then $c'$ crosses the same (empty) set of chords as $c$.  It follows that $\alpha$ has a chord $c'$ which is traversed by two {\em different} cycles.  Thus $\alpha$ has at least two cycles contradicting the assumption that $\alpha$ is a $0$-filling system.

Now since $\alpha_c$ is a $1$-filling system we may orient it and take its boundary.  Next we claim
\begin{equation}
\label{eq:connection_relation}
\partial \alpha_c = \pm \alpha.
\end{equation}
By definition $ \partial \alpha_c$ is a linear combination of the $0$-filling systems that one can obtain by removing a single chord from $\alpha_c$.  No chords of $\alpha_c$ cross the chord $c$ so any chord diagram $\beta$ obtained from $\alpha_c$ by removing a chord other than $c$ cannot have a single cycle which traverses $c$ twice.   Such a chord diagram $\beta$ must have two different cycles which traverse $c$, so $\beta$ cannot be a $0$-filling system.  Thus $\alpha$ is the only $0$-filling system that one may get by removing a chord from $\alpha_c$.  Equation~(\ref{eq:connection_relation}) follows establishing the proposition.
 \end{proof}

Now we will introduce a subset of the $0$-filling systems whose classes generate $\St(\surface_g^1)$ and for which, by Proposition~\ref{prop:connected}, the class of each element, save one, is trivial.
\begin{figure}[!ht]
\begin{center}
\includegraphics[scale=0.4]{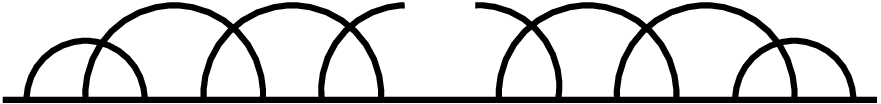}
\put(-92,4){\large$\cdots$}
\end{center}
\caption{A connected component of a salient chord diagram\label{fig:salient}}
\end{figure}
A chord diagram will be called \emph{salient} if each of its connected components is of the form given in 
Figure~\ref{fig:salient}.
Of course $\phi_0$ from Figure~\ref{fig:st_gen} is the unique connected, salient $0$-filling system.

\begin{proposition}
\label{prop:salient}
 $\St(\surface_g^1)$ is generated by salient $0$-filling systems.
\end{proposition}

\begin{proof}
For $n \geq 0$ a chord diagram will be said to have a \emph{salient tail of length $n$} if a neighborhood of some segment in the boundary of the diagram is homeomorphic to the neighborhood pictured in Figure~\ref{fig:salient_tail}.  
\begin{figure}[!ht]
\begin{center}
\includegraphics[scale=0.4]{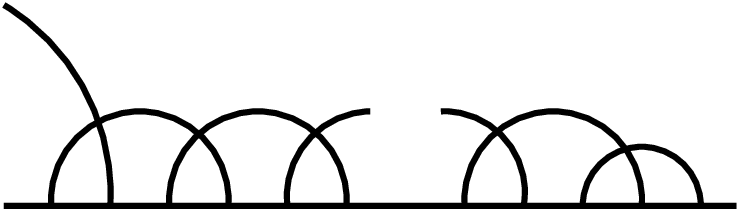}
\put(-120,22){$\overset{\textstyle n \mbox{ chords}}{\overbrace{\hspace{1.55in}}}$}
\put(-71,4){\large$\cdots$}
\end{center}
\caption{A salient tail of length $n$ in a chord diagram\label{fig:salient_tail}}
\end{figure}
Given a $0$-filling system $\alpha$ with a salient tail of length $n$, we can add a chord $c$ to the far right of the salient tail to get a $1$-filling system $\alpha_c$ with a salient tail of length $n+1$.   There is some $m$ with $0 \leq m \leq 2g$ and there are $0$-filling systems $\beta_i$  for $ 1 \leq i \leq m$ such that
\begin{equation*}\partial \alpha_c = \pm \alpha + \sum_{i=1}^{m} \pm \beta_i.\end{equation*}
Notice that if any single chord in the salient tail of $\alpha_c$ other than $c$ is removed from $\alpha_c$ then the resulting chord diagram is disconnected, and so by Proposition~\ref{prop:connected} the class of that $0$-filling system is trivial in $\St(\surface_g^1)$.  Thus for every $i$ with $1 \leq i \leq m$ the class of the $0$-filling system $\beta_i$ is either trivial in $\St(\surface_g^1)$ or the chord diagram for $\beta_i$ has a salient tail of length $n+1$.  Iterating this process recursively no more than $2g$ times we may write the class of any $0$-filling system as a linear combination of classes of salient $0$-filling systems.
 \end{proof}

With Propositions~\ref{prop:salient} and \ref{prop:connected} established, the proof of Theorem~\ref{thm:single_gen} is quite short.
\begin{proof}[Proof of Theorem~\ref{thm:single_gen}]
 By Proposition~\ref{prop:salient}, $\St(\surface_g^1)$ is generated by salient $0$-filling systems.  The element $\phi_0 \in \fsys_0$ is a representative of the unique $\Mod(\surface_g^1)$-orbit of connected, salient $0$-filling systems.  Thus by Proposition~\ref{prop:connected}, $[\phi_0]$ generates $\St(\surface_g^1)$.

For the closed surface without a marked point $\surface_g$ we have $\St(\surface_g) \cong \St(\surface_g^1)$ as $\Mod(\surface_g^1)$-modules, so $[\phi_0]$ generates $\St(\surface_g)$ as a $\Mod(\surface_g^1)$-module.  On the other hand, the $\Mod(\surface_g^1)$-module structure of $\St(\surface_g)$ factors through $\Mod(\surface_g)$, so $[\phi_0]$ generates $\St(\surface_g)$ as a $\Mod(\surface_g)$-module.
 \end{proof}

\begin{remark}
As a consequence of Theorem~\ref{thm:single_gen} there is a left ideal
\begin{equation*}
\mathcal{J} \subset \Z \Mod(\surface_g)
\end{equation*}
such that the sequence of $\Mod(\surface_g)$-modules
\begin{equation}
 0 \to \mathcal{J} \to \Z \Mod(\surface_g) \to \St(\surface_g) \to 0
\end{equation}
is exact.  In theory, for a fixed genus $g$ the presentation in Proposition~\ref{prop:stein_res} gives enough information to calculate a finite generating set for $\mathcal{J}$.  
A more ``geometric'' understanding of $\mathcal{J}$ would certainly advance our understanding of $\St(\surface_g)$.  At minimum one would like to know the stabilizer in $\Mod(\surface_g)$ of the class $[\phi_0] \in \St(\surface_g)$.
\end{remark}


\subsection{Spheres in the curve complex}
\label{sec:curv_sphere}

As illustrated above it is useful to view the Steinberg module as the reduced homology of the arc complex at infinity. On the other hand, the curve complexes $\CC (\surface_g)$ and $\CC (\surface_g^1)$ are more widely studied so we will now explain Harer's homotopy equivalence  $\Psi$ between the arc complex at infinity and the curve complex.  We will use this map to get an explicit nontrivial $2$-sphere in the curve complex of the surface of genus $2$ (see Proposition~\ref{prop:2-gen} below) which will motivate Conjecture~\ref{conj:nontriv} below.  We point out that results of Penner-McShane in \cite{PM2007} provide an alternative approach to converting a $0$-filling system into a sphere in the curve complex.

 Harer \cite[Theorem~3.4]{Harer1986} gives a homotopy equivalence from the arc complex at infinity to the curve complex.  Harer's map
\begin{equation*}\Psi: \AC_\infty^{\circ \circ}(\surface_g^1) \to \CC^{\circ}(\surface_g^1)\end{equation*}
is simplicial, where $\AC_\infty^{\circ \circ}(\surface_g^1)$ denotes the second barycentric subdivision of $\AC_\infty(\surface_g^1)$ and $\CC^{\circ}(\surface_g^1)$ denotes the first barycentric subdivision of $\CC(\surface_g^1)$.  The map $\Psi$ is defined on the vertices of $\AC_\infty^{\circ \circ}(\surface_g^1)$, as follows.  A vertex $v$ of $\AC_\infty^{\circ \circ}(\surface_g^1)$ is a nested sequence of non-filling arc systems $\beta_1 \subset \beta_2 \subset \cdots \subset \beta_k$.  
For $1 \leq i \leq k$ if one removes a small open regular neighborhood of the union of the arcs in $\beta_i$ from $\surface_g^1$, one is left with a surface $\surface(i) \subset \surface_g^1$ with nonempty boundary.  After omitting redundancies and trivial curves, the boundary components of $\surface(i)$ give a curve system $C_i$ in $\surface_g^1$.  We define $\Psi(v) := \bigcup_{i=1}^k C_i$ again omitting redundancies.  In fact, $\Psi(v)$ is a curve system since if $\beta_i \subset \beta_j $ then we can arrange that $\surface(i) \supset \surface(j)$ and hence the curves of $C_i$ can be taken to be disjoint from the curves of $C_j$.  Of course $\Psi(v)$ is a vertex of $\CC^{\circ}(\surface_g^1)$ whose vertices are curve systems.  We then extend $\Psi$ simplicially; that is, the simplex with vertices $v_1, \cdots, v_n$ is mapped linearly to the simplex with vertices $\Psi(v_1), \cdots, \Psi(v_n)$.

The $0$-filling system $\phi_0$ in Figure~\ref{fig:st_gen} gives an arc system whose class in $\HH_{2g-1}(\AC / \AC_\infty ; \Z )$ is nontrivial.  From the long exact sequence for the homology of the pair of spaces $(\AC,\AC_\infty)$ we have the isomorphism
\begin{equation*}\HH_{2g-1}(\AC / \AC_\infty ; \Z ) \stackrel{\partial}{\cong} \widetilde{\HH}_{2g-2}( \AC_\infty ; \Z ).\end{equation*}
In $\AC (\surface_g^1)$ the arc system for $\phi_0$ gives a $(2g-1)$-simplex all of whose proper subfaces are contained in $\AC_\infty (\surface_g^1)$.  The class of $\partial[\phi_0] \in \widetilde{\HH}_{2g-2}( \AC_\infty ; \Z )$ is represented by the boundary of this $(2g-1)$-simplex.

For the surface $\surface_2^1$  we may compute the image of $\partial\phi_0$ under $\Psi$ directly. Figure~\ref{fig:psi_image} gives the image of the vertices of the first barycentric subdivision of $\partial\phi_0$.
\begin{figure}[!ht]
\begin{center}
\resizebox{0.75\textwidth}{!}{%
\includegraphics[scale=0.35]{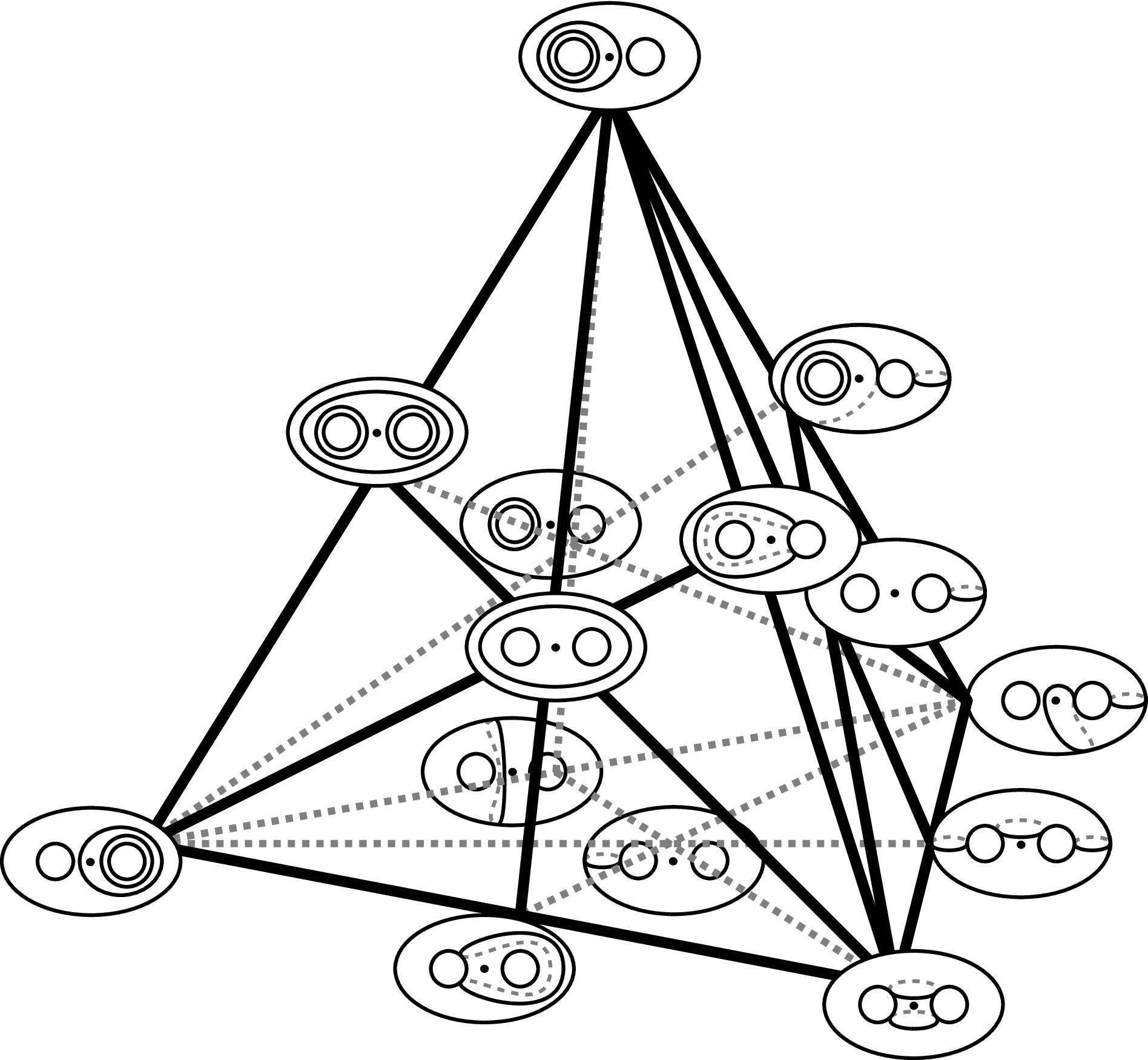}
}
\end{center}
\caption{A nontrivial sphere in the barycentric subdivision $\CC^{\circ}(\surface_2^1)$ of the curve complex\label{fig:psi_image}}
\end{figure}
One may easily fill in  the images for vertices in the second barycentric subdivision by taking unions of the curve systems at the vertices of the first subdivision.  The sphere in Figure~\ref{fig:psi_image} is nontrivial, but it is slightly unsatisfying in that it is specified in $\CC^{\circ}(\surface_2^1)$ and not $\CC(\surface_2^1)$. Figure~\ref{fig:sphere} and Proposition~\ref{prop:2-gen}  below provide a homotopic sphere with a nicer description.  The reader may recognize the shape in Figure~\ref{fig:sphere} to be the boundary of the dual of the $3$-dimensional associahedron (see \cite{Lee1989}).
\begin{figure}[!ht]
\begin{center}
\resizebox{0.75\textwidth}{!}{%
\includegraphics[scale=0.35]{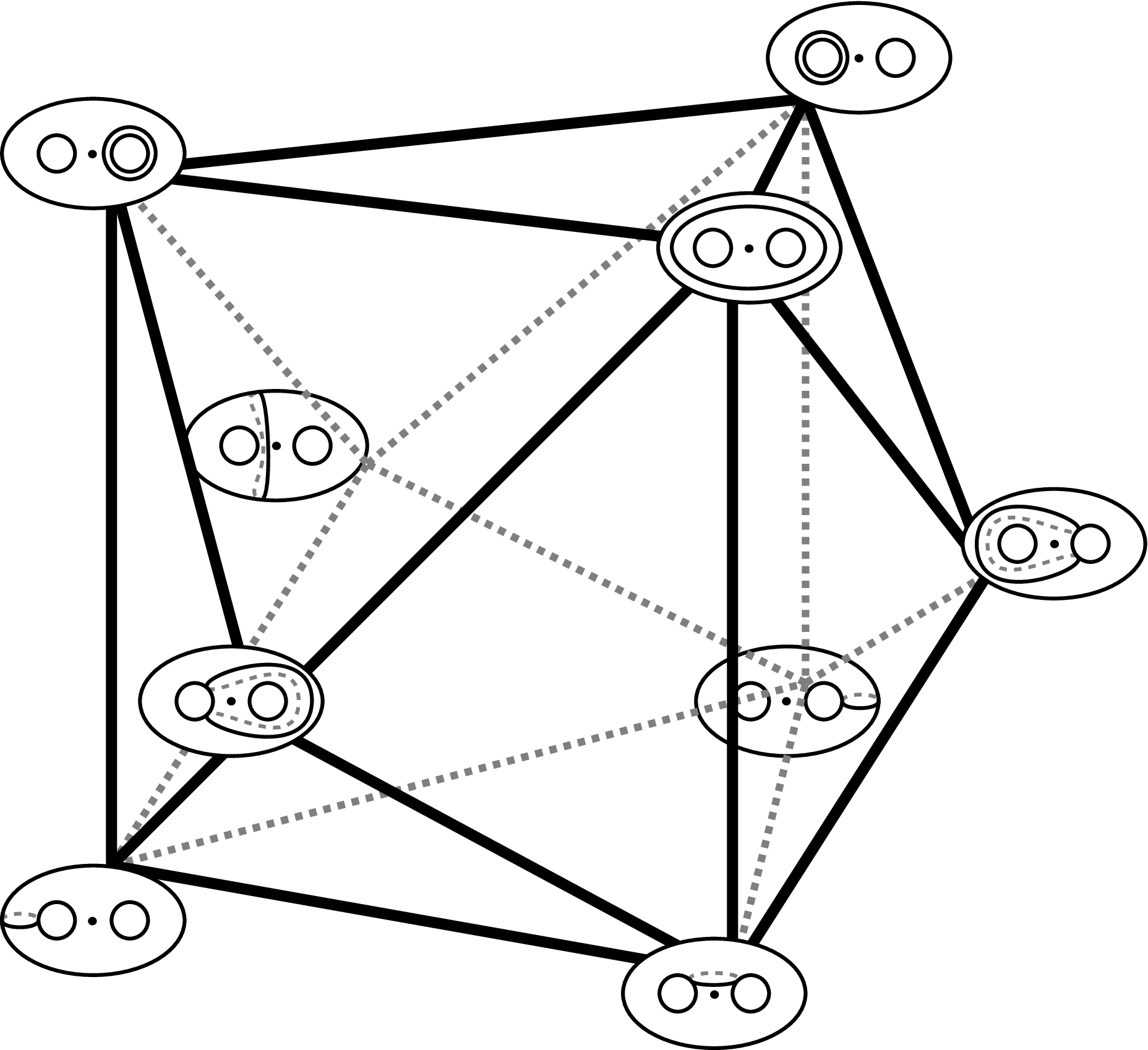}
}
\end{center}
\caption{This is a homologically nontrivial sphere in $\CC(\surface_2^1)$.  Forgetting the marked point gives a homologically nontrivial sphere in $\CC(\surface_2)$.\label{fig:sphere}}
\end{figure}
\begin{proposition}
\label{prop:2-gen}
The $\Mod(\surface_2^1)$-orbit of the homology class of the $2$-sphere in $\CC(\surface_2^1)$ pictured in Figure~\ref{fig:sphere} generates $\HH_2(\CC(\surface_2^1);\Z)$. Forgetting the marked point gives a $2$-sphere in $\CC(\surface_2)$ whose $\Mod(\surface_2)$-orbit generates $\HH_2(\CC(\surface_2);\Z)$.
\end{proposition}

\begin{proof}

The arc system for the generator $[\phi_0] \in \St(\surface_2^1)$ from Theorem~\ref{thm:single_gen} is pictured in Figure~\ref{fig:chord_diagram}.  One may use Harer's map directly to show that the vertices in the first barycentric subdivision of $\partial\phi_0$ are mapped under $\Psi$ to the curve systems pictured in Figure~\ref{fig:psi_image}.  We will give a simplicial sphere in the unbarycentricly subdivided curve complex which is homotopic to the sphere pictured in Figure~\ref{fig:psi_image}.  One may construct a homotopy equivalence
\begin{equation*}f: \CC(\surface_2^1) \to \CC(\surface_2^1)\end{equation*}
as follows.  Let $Y$ be the set of exactly those curves appearing in Figure~\ref{fig:psi_image}.  Fix some linear order on $Y$ so that each of the curves that appear in Figure~\ref{fig:sphere} is greater than all of those that do not. Let $f:\CC(\surface_2^1) \to \CC(\surface_2^1)$ be the
map satisfying the following three properties:
\begin{enumerate}
 \item The map $f$ fixes the vertices of $\CC(\surface_2^1)$.
\item For any simplex $\sigma$ of $\CC(\surface_2^1)$ whose vertices are disjoint from $Y$ the map $f|_\sigma:\sigma \to \sigma$ is the identity.
\item If at least one vertex of the simplex $\sigma$ of $\CC(\surface_2^1)$ is in the set $Y$ then $f|_\sigma:\sigma \to \sigma$ is the map which is linear on the barycentric subdivision of $\sigma$, fixes the vertices of $\sigma$ and sends the barycenter of $\sigma$ to the greatest vertex of $\sigma$ in the order on $Y$.
\end{enumerate}
The map $f: \CC(\surface_2^1) \to \CC(\surface_2^1)$ is homotopic to the identity.  In fact, sending each point of a simplex to weighted averages of itself and its image under $f$ gives the homotopy between $f$ and the identity.  It follows that the sphere $f \Psi(\partial \phi_0)$ is homotopic to $\Psi(\partial \phi_0)$. Some of the vertices in Figure~\ref{fig:psi_image} coalesce under $f$ and we get the simpler homotopic sphere pictured in Figure~\ref{fig:sphere}.  We may then conclude that the $\Mod(\surface_2^1)$-orbits of sphere in Figure~\ref{fig:sphere} generate $\HH_2(\CC(\surface_2^1);\Z)$. Forgetting the marked point gives a $2$-sphere in $\CC(\surface_2)$ whose $\Mod(\surface_2)$-orbit generates $\HH_2(\CC(\surface_2);\Z)$ \cite[Lemma~3.6]{Harer1986}.
 \end{proof}

The number of vertices in the boundary of the first barycentric subdivision of the $(2g-1)$-simplex is $2^{2g}-2$.  Hence, for large $g$ the direct approach to producing homologically nontrivial spheres in $\CC(\surface_g^1)$ given in the proof of Proposition~\ref{prop:2-gen} is impractical.  Moreover, the asymmetry of the arc system for $\phi_0$ is likely to yield a sphere in $\CC(\surface_g^1)$ which is difficult to even describe in any concise form.

\begin{figure}[!ht]
\begin{center}
\resizebox{0.65\textwidth}{!}{%
\includegraphics[scale=0.5]{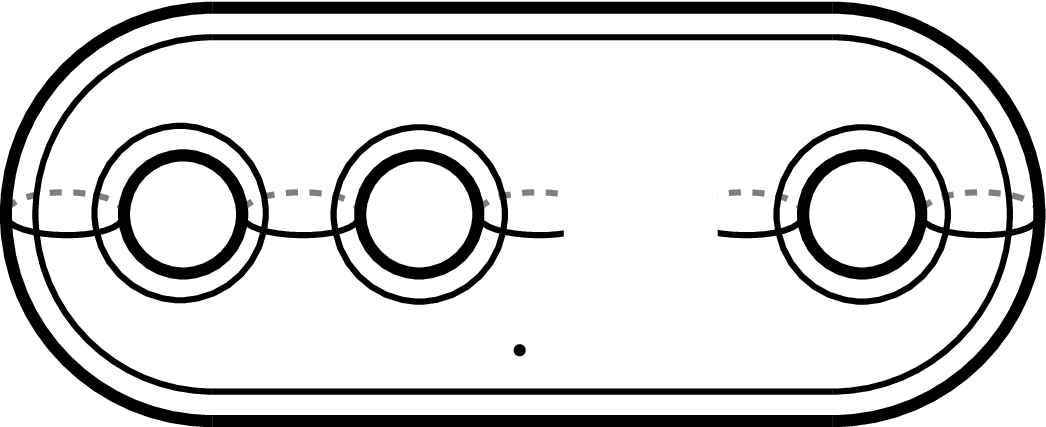}
\put(-108,47){\Large$\cdots$}
}
\end{center}
\caption{Using these $2g+2$ curves one can construct a map from the boundary of the dual of the $(2g-1)$-dimensional associahedron into $\CC(\surface_g^1)$.  Is that map homologically nontrivial?\label{fig:assoc}}
\end{figure}

Using the curves in Figure~\ref{fig:assoc} there is a simple construction of a map of a $(2g-2)$-sphere into $\CC(\surface_g^1)$.  Let $\Theta$ be the boundary of the dual of the $(2g-1)$-dimensional associahedron which we now define.  Fix a regular $(2g+2)$-gon $K$.  By definition $\Theta$ is the simplicial complex whose vertices are diagonals of $K$ and whose simplices are given by sets of disjoint diagonals in $K$. It is well-known that $\Theta$ is a $(2g-1)$-dimensional sphere \cite{Lee1989}.  Associate the vertices of $K$ with the $2g+2$ curves in Figure~\ref{fig:assoc} so that adjacent vertices of $K$ correspond to intersecting curves.  Each pair of nonintersecting curves in Figure~\ref{fig:assoc} corresponds to a diagonal of $K$.  Let
\begin{equation*}q':\Theta \to \CC^\circ(\surface_g^1)\end{equation*}
be the simplicial map sending each pair of nonintersecting curves in Figure~\ref{fig:assoc} to the curve system (with at most $4$ nontrivial curves) consisting of the boundary curves of the surface got by removing an open regular neighborhood of the union of all the other curves.  For each vertex $v$ of $\Theta$ choose a curve $c_v \in q'(v)$ and let
\begin{equation*}q:\Theta \to \CC(\surface_g^1)\end{equation*}
be the simplicial map sending $v$ to $c_v$.  One may construct a homotopy like the one in the proof of Proposition~\ref{prop:2-gen} above to show that $q$ and $q'$ are homotopic and hence the homotopy class of $q$ is independent of the choices of $c_v \in q'(v)$.
\begin{conjecture}
\label{conj:nontriv}
 For $g \geq 1$ the class $[q] \in \widetilde{\HH}_{2g-2}(\CC(\surface_g^1);\Z)$ is nontrivial. (When $g=1$ the proper picture for Figure~\ref{fig:assoc} consists of $2$ pairs of parallel curves.)
\end{conjecture}

By Proposition~\ref{prop:2-gen} above the conjecture holds for $g=2$.  It also holds\footnote{Of course the mapping class groups of the surfaces of genus $1$ and $2$ are somewhat exceptional. (For instance they have nontrivial centers \cite[Theorem~7.5.D]{Ivanov2002}.)} for $g=1$.  

\section{Questions for further study}
\label{sec:qandn}

Let $n \geq 2$. The negative of the identity matrix $-I \in \SL(n,\Z)$ stabilizes every subspace of $\Q^n$, so $-I$ is in the kernel of 
the action of $\SL(n,\Z)$ on its Steinberg module $\St(n)$.  Consequently the action of $\SL(n,\Z)$ on $\St(n)$ factors though an action of the simple group $\PSL(n,\Q)$ (see \cite[Theorem~9.3]{Lang2002}).  Simple groups must either act trivially or faithfully.   Since $\St(n)$ is a nontrivial  $\SL(n,\Z)$-module it must also be a nontrivial and hence faithful $\PSL(n,\Q)$-module.  It follows that $\St(n)$ is a faithful $\PSL(n,\Z)$-module.

As we have seen in Corollary~\ref{cor:factors}, for $g \geq 2$ the kernel of the action of $\Mod(\surface_g^1)$ on the Steinberg module contains the infinite point pushing subgroup $P < \Mod(\surface_g^1)$.  For the closed surface without a marked point there is some hope that the action of the mapping class group (modulo its center) on the Steinberg module is faithful.
When $g \in \{1 ,2 \}$, the hyper\-elliptic involution generates the center  $Z=Z(\Mod(\surface_g))$ of the mapping class group and acts trivially on the curve complex.  For $g \geq 3$ the center $Z=Z(\Mod(\surface_g))$ is trivial \cite[Theorem~7.5.D]{Ivanov2002}.  Thus for $g \geq 1$ the action of $\Mod(\surface_g)$ on $\St(\surface_g)$ factors through $\Mod(\surface_g) / Z$.

\begin{question}
For $g \geq 1$ let $Z = Z(\Mod(\surface_g))$ be the center of $\Mod(\surface_g)$. Is $\St(\surface_g)$ a faithful $\Mod(\surface_g) / Z$-module? If not what is the kernel of the action?
\end{question}

The Solomon-Tits Theorem ({\it cf.} \cite{Solomon1969}, \cite[\S IV.5 Theorem 2]{Brown1998}) shows that the Tits building $\TB(n,\Q)$ has the homotopy type of a wedge of spheres, that its reduced homology is a cyclic module (over the associated $\Q$-group), and futher gives a $\Z$-basis for the reduced homology.  Harer has shown that the curve complex has the homotopy type of a wedge of spheres.  Theorem~\ref{thm:single_gen} above shows that the reduced homology of the curve complex is a cyclic $\Mod(\surface_g)$-module. The full ``Solomon-Tits Theorem'' for the mapping class group should also provide the following.

\begin{problem}
 Give a $\Z$-basis for $\St(\surface_g) \cong \St(\surface_g^1)$.
\end{problem}

\medskip
\noindent
{\bf Steinberg modules for automorphism groups of free groups.}
Let $F_n$ be the free group on $n$ generators, and let $\Aut(F_n)$ and $\Out(F_n)$ denote its automorphism group and outer automorphism group respectively (see \cite{Vogtmann2002} for a survey of these groups).  Many of the results on the homological structure of the mapping class group have analogs for $\Out(F_n)$ and $\Aut(F_n)$.   For example, Bestvina and Feighn \cite{BF2000} have shown that $\Out(F_n)$ and $\Aut(F_n)$ are virtual duality groups.  To date, the module structure of the dualizing modules for (torsion-free, finite index subgroups of) these groups has not been studied in depth.

In the case of $\Aut(F_n)$ Hatcher and Vogtmann \cite{HV1998} have proposed a likely candidate for this dualizing module.  They define a simplicial complex based on the poset of free factors of $F_n$ and show that it has the homotopy type of a wedge of spheres.  They define its reduced homology group to be the Steinberg module of $\Aut(F_n)$, and ask if this module provides the dualizing module for finite index, torsion-free subgroups of $\Aut(F_n)$.  It is possible that the approach of the current paper might be used to address this question.

\bigskip

\noindent
{\sc Department of Math\\
Ohio State University\\
231 W 18th Ave.\\
Columbus, OH 43210 \\}
\medskip
{\tt \verb|broaddus@math.ohio-state.edu|}

\end{document}